\documentclass[10pt]{amsart}
\usepackage{latexsym}
\usepackage{amsfonts}
\usepackage{amssymb}
\usepackage{amsmath}
\usepackage{mathrsfs}
\usepackage{hyperref}
\usepackage{tikz}
\usetikzlibrary{matrix,arrows}
\tikzstyle{shaded}=[fill=red!10!blue!20!gray!30!white]
\tikzstyle{shaded line}=[double=red!10!blue!20!gray!30!white, double distance=1.5mm, draw=black]
\tikzstyle{unshaded}=[fill=white]
\tikzstyle{unshaded line}=[double=white, double distance=1.5mm, draw=black]
\tikzstyle{Tbox}=[circle, draw, thick, fill=white, opaque,]
\tikzstyle{empty box}=[circle, draw, thick, fill=white, opaque, inner sep=2mm]
\tikzstyle{background rectangle}= [fill=red!10!blue!20!gray!40!white,rounded corners=2mm] 
\tikzstyle{on}=[very thick, red!50!blue!50!black]
\tikzstyle{off}=[gray]

\tikzstyle{traces}=[scale=.2, inner sep=1mm]
\tikzstyle{quadratic}=[scale=.4, inner sep=1mm, baseline]
\tikzstyle{annular}=[scale=.7, inner sep=1mm, baseline]
\tikzstyle{make triple edge size}= [scale=.4, inner sep=1mm,baseline] 
\tikzstyle{icosahedron network}=[scale=.3, inner sep=1mm, baseline]
\tikzstyle{ATLsix}=[scale=.25, baseline]
\tikzstyle{TL12}=[scale=.15,baseline]
\tikzstyle{PAdefn}=[scale=.7,baseline]
\tikzstyle{TLEG}=[scale=.5,baseline]

\newcommand{\newsection}[1]{\setcounter{equation}{0}
  \setcounter{definition}{0}
\section{#1}}

\newtheorem{definition}{Definition}[section] 
\newtheorem{theorem}[definition]{Theorem}

\newtheorem{proposition}[definition]{Proposition}
\newtheorem{corollary}[definition]{Corollary} 

\newtheorem{remark}[definition]{Remark}

\newcommand{\AAZ}{{\mathcal A}_{\hbar{}}^{\infty}}
\def \qed 
{
  \mbox{}\hfill $\Box$\vspace{1ex}}

\title{Yang- Mills on Quantum Heisenberg Manifolds}
\author{Partha Sarathi Chakraborty}
\address{The Institute of Mathematical Sciences, CIT Campus, Taramani, Chennai
600113}
\email{parthac@imsc.res.in}
\author{Satyajit Guin}
\address{The Institute of Mathematical Sciences, CIT Campus, Taramani, Chennai
600113}
\email{gsatyajit@imsc.res.in}
\date{\today}

\begin{document}
\rmfamily
\begin{abstract} In the noncommutative geometry program of Connes there are two variations of the concept of Yang-Mills action functional. We show that for the quantum Heisenberg manifolds they agree.
\end{abstract}
\maketitle
\section {Introduction}
Quantum Heisenberg manifolds (QHM) were introduced by Rieffel in \cite{RI1} as strict deformation quantization of Heisenberg manifolds. He introduced a parametric family of deformations and for generic parameter values these are simple $C^*$-algebras with an ergodic action of the Heisenberg group of $3 \times 3$ upper triangular matrices with ones on the diagonal. They admit a unique invariant trace. Connes has showed in \cite{Co1} that whenever one has a $C^*$-dynamical system with dynamics governed by a Lie group and an invariant trace one can extend the basic notions of geometry. Leter Connes and Rieffel (\cite{CR}) has introduced the concepts of Yang-Mills action functional and quantum Heisenberg manifold presents an ideal case for such considerations. Recently Kang \cite{KANG} has studied Yang-Mills for the QHM. However the popular formulation of noncommutative geometry today is through spectral triples. In this approach as well Connes (\cite{CON})  defined the concept of the Yang-Mills action functional. Now starting with a $C^*$-dynamical system with an invariant trace there is a general prescription that produces a candidate for a spectral triple, but there is no general theorem and in each case one has to verify the self-adjointness and the compact resolvent of the Dirac operator.  It was shown in \cite{CS} that in the case of QHM the general principle gives rise to an honest spectral triple. A natural question in this context is whether even in this case these two notions of YM coincide and that is the content of this paper. We show that the notions agree  in the context of QHM. This parallels proposition 13 in the last chapter of \cite{CON} where similar results were obtained for noncommutative two torus.

\newsection{The Quantum Heisenberg Algebra}
Notation: for $x \in \mathbb{R}$, $e(x)$ stands for $e^{2 \pi  ix}$ where $i=\sqrt{-1}.$
\begin{definition} \rm
For any positive integer $c$, let $S^c$ denote the space of  smooth  functions 
$\Phi : \mathbb{R} \times \mathbb{T} \times \mathbb{Z}  \rightarrow C$ such that 
\begin{itemize}
\item $\Phi (x+k,y,p)=e(ckpy) \Phi(x,y,p) $ for all $k \in \mathbb{Z}$,
\item  for every polynomial $P$ on $\mathbb{Z}$ and every partial differential operator $ \widetilde{X}=\frac{\partial^{m+n}}{\partial x^m \partial y^n} $ on $\mathbb{R} \times \mathbb{T}$ the function $ P(p)(\widetilde{X} \Phi) (x,y,p)$ is bounded on $ K \times \mathbb{Z}$ for any compact subset $K$  of $\mathbb{R} \times \mathbb{T}$.
\end{itemize}
For each  $\hbar,\mu,\nu  \in \mathbb{R},\mu^2 + \nu^2 \ne 0$, let ${\mathcal{A}}^{\infty}_{\hbar} $ denote $S^c$ with product and involution defined by
\begin{eqnarray}
\label{1}
(\Phi \star \Psi)(x,y,p)= \sum_q \Phi(x-\hbar (q-p) \mu ,y-\hbar (q-p) \nu,q) \Psi (x- \hbar q \mu,y-\hbar q \nu,p-q) \end{eqnarray}
\begin{eqnarray} \label{2} \Phi^*(x,y,p)= \bar{\Phi}(x,y,-p). \end{eqnarray}
Then, $\pi : {\mathcal{A}}^{\infty}_{\hbar}\rightarrow \mathcal{B}(L^2( \mathbb{R} \times \mathbb{T} \times \mathbb{Z})) $ given by
\begin{eqnarray} \label{3}
(\pi (\Phi) \xi)(x,y,p)= \sum_q \Phi ( x - \hbar ( q-2p)\mu,y- \hbar (q-2p) \nu,q) \xi ( x,y,p-q)
\end{eqnarray} gives a faithful representation of the involutive  algebra $ \AAZ$. 
${\mathcal{A}}^{c,\hbar}_{\mu,\nu}=$ norm closure of $\pi( \AAZ )$ is called the Quantum Heisenberg Manifold.
\end{definition}
We will identify $\AAZ$ with $\pi( \AAZ) $ without any mention.
Since we are  going to work with fixed parameters $ c, \mu,\nu, \hbar $ we will drop them altogether and denote ${\mathcal{A}}^{c,\hbar}_{\mu,\nu}$ simply by $\mathcal{A}_\hbar$ here the subscript remains merely as a reminiscent of  Heisenberg only to distinguish it from a general algebra.


 {\bf Action of the heisenberg group:} Let $c$ be a positive integer. Let us consider the  group structure on $G=\mathbb{R}^3=\{(r,s,t): r,s,t \in \mathbb{R}\}$ given by the multiplication 
\begin{eqnarray} \label{G}(r,s,t)(r',s',t')=(r+r',s+s',t+t'+csr'). \end{eqnarray} Later we will give an explicit isomorphism between $G$ and $H_3$, the Heisenberg group of $3 \times 3$ upper triangular matrices with real entries and ones on the diagonal. Through this identification we can identify $G$ with the Heisenberg group.
For $ \Phi \in S^c, (r,s,t) \in \mathbb{R}^3 \equiv G$,  
\begin{eqnarray}
\label{4}
(L_{(r,s,t)} \phi )(x,y,p)=e(p(t+cs(x-r)))\phi(x-r,y-s,p) \end{eqnarray}
extends to an ergodic action of the Heisenberg group on $ {\mathcal{A}}^{c,\hbar}_{\mu,\nu}$.


{\bf The Trace:}
 The linear functional $\tau : \AAZ \rightarrow \mathbb{C}$, given by $\tau (\phi)= \int^1_0 \int_{\mathbb{T}} \phi (x,y,0) dx dy $
is invariant under the Heisenberg group action. So, the group action can be lifted to $L^2( \AAZ)$.
We will denote the action at the Hilbert space level by the same symbol.

\newsection{Yang-Mills in the dynamical system approach}
In (\cite{CR}) Connes and Rieffel introduced Yang-Mills functional in the setting of $C^*$-dynamical systems. We will recall their definition in the context of QHM. Here the dynamics is governed by the Lie group $G$. We can identify $G$ with $H_3$ through the isomorphism that identifies $(r,s,t) \in G$ with the matrix $\begin{pmatrix} 1 & cs & t \cr 0 & 1 & r \cr 0 & 0 & 1 \end{pmatrix} $. Let $\mathfrak{g}$ be the Lie-algebra of $G$. We can identify  $\mathfrak{g}$ with the Lie-algebra of  $H_3$, which is given by $3 \times 3$ upper triangular matrices with real entries with zeros on the diagonal. Fix a real number $\alpha$ greater than one. This number will remain fixed throughout and we will comment about it later. In this approach one has to fix an inner product structure on the Lie algebra of the underlying Lie group and in our case we do so  by declaring  the following basis,
\begin{eqnarray}\label{basis}
X_1=\left(\begin{matrix} 0&0&0 \cr 0&0&1 \cr 0& 0&0\cr \end{matrix}\right),
X_2=\left(\begin{matrix} 0&c&0 \cr 0&0&0\cr 0& 0&0\cr\end{matrix}\right),
X_3=\left(\begin{matrix} 0&0&c\alpha  \cr 0&0&0 \cr 0& 0&0\cr\end{matrix}\right) 
\end{eqnarray}
as orthonormal. Their Lie  bracket is given by,
\begin{eqnarray}\label{Lie}
 [X_1,X_3]=[X_2,X_3]=0, [X_1,X_2]=-\frac{1}{\alpha}X_3.
\end{eqnarray}
The exponential map from $\mathfrak{g}$ to $G$ acts on these elements as follows 
\begin{eqnarray*} exp(rX_1)=(r,0,0), exp(sX_2)=(0,s,0), \mbox{ and } exp(tX_3)=(0,0,c \alpha t).
\end{eqnarray*}
 For $X \in \mathfrak{g}$, let $d_X$ be the 
derivation of $\AAZ$ given by $d_X(a)=\frac{d}{dt}\mid_{t=0}L_{exp(tX)}(a)$. Let us denote the $d_{X_j}$'s, for $j=1,2,3$ by $d_j$. Then they are given by 
\begin{eqnarray}
\label{derivation-I}
d_1(f)&=& - \frac {\partial f } {\partial x}, \\
\label{derivation-II} d_2 (f) &=& 2 \pi i c p x f(x,y,p) -  \frac { \partial f } {\partial y}, \\
\label{derivation-III} d_3 (f) &=& 2 \pi i p c \alpha f(x,y,p). 
\end{eqnarray}
We now recall the  Hermitian structure on finitely generated projective modules. This is needed to define the Yang-Mills action functional. Let $\mathcal{H}$ be a Hilbert space and $\mathcal{A}$ be a unital involutive subalgebra of $\mathcal{B}(\mathcal{H})$, the algebra of bounded operators on $\mathcal{H}$, closed under holomorphic function calculus.
Let $\mathcal{E}$ be a finitely generated projective $\mathcal{A}$ module.
Define $\mathcal{E}^*$ as the space of $\mathcal{A}$ linear mappings from $\mathcal E$ to $\mathcal{A}$. Clearly $\mathcal{E}^*$ is a right $\mathcal{A}$ module.
\begin{definition}\rm \label{Hermitian-1}
A  Hermitian structure on $\mathcal{E}$ is an $\mathcal{A}$-valued positive-definite inner product 
$\langle \, \, , \, \rangle_{\mathcal{A}} $ such that, 
\begin{enumerate}
\item [(a)] $\langle \xi , \xi' \rangle_{\mathcal{A}}^* = \langle \xi' , \xi \rangle_{\mathcal{A}}\, , \, \, \, \forall \, \xi , \xi' \in \mathcal{E}$.
\item [(b)] $\langle \xi , \xi'. a \rangle_{\mathcal{A}} =  (\langle \xi , \xi' \rangle_{\mathcal{A}}) .a\, , \, \, \, \forall \, \xi , \xi' \in \mathcal{E},\, \, \forall \, a \in \mathcal{A}$.
\item [(c)] The map $\xi \longmapsto \Phi_\xi$ from $\mathcal{E}$ to $\mathcal{E}^*\,$, given by $\Phi_\xi(\eta) = \langle\xi,\eta\rangle_{\mathcal{A}} \, , \, \forall \eta \in \mathcal{E}\,$, gives an $\mathcal{A}$-module isomorphism between $\mathcal{E}$ and $\mathcal{E}^*$. This property will be referred as the self-duality of $\mathcal E$.
\end{enumerate}
\end{definition}
Let $\mathcal{E}$ be a finitely generated projective $\AAZ$ module with a hermitian structure. A connection is a map 
\begin{eqnarray}\label{connection-I}
          \nabla &:& \mathcal{E} \rightarrow \mathcal{E} \otimes {\mathfrak{g}}^*, \mbox{ such that }\\
\nabla_X(\xi.a)& =&  \nabla_X(\xi).a + \xi .d_X(a) \forall \xi \in \mathcal{E}, \, \forall a \in {\AAZ}.
\end{eqnarray}
We shall say that $\nabla$ is compatible with respect to the Hermitian structure on $\mathcal{E}$ 
iff~:
\begin{eqnarray}\label{compatibility-I}
\langle \nabla_X \, \xi\, , \xi' \rangle_{\mathcal{A}} \, + \, \langle \xi\, , \nabla_X \, \xi' 
\rangle_{\mathcal{A}} = d_X (\langle\, \xi , \xi'\, \rangle_{\mathcal{A}}) \, , \quad \forall \, \xi\, , 
\xi' \in \mathcal{E},\, \, \forall \, X \in \mathfrak{g} \, .
\end{eqnarray}
We will denote the set of compatible connections by $C(\mathcal{E})$.
The  curvature  $\Theta_\nabla $ of a  connection $\nabla$ is the alternating bilinear $End(\mathcal{E})$-valued form on $\mathfrak{g}$ defined by,
\begin{eqnarray*}
\Theta_\nabla (X \wedge Y) = [\nabla_X , \nabla_Y] - \nabla_{[X,Y]}, \forall X,Y \in  \mathfrak{g}.
\end{eqnarray*}
\begin{proposition}\label{compatible-connection-I} Let $\mathcal E$ be a finitely generated projective $\AAZ$ module. Then  the space $C(\mathcal{E})$ 
of compatible  connections is given by triples of linear maps $\nabla_j : \mathcal{E} \rightarrow \mathcal {E}, j=1,2,3 $ such that
\begin{align} \label{compatible-connection-I1}
\nabla_j(\xi.a) &= \nabla_j(\xi).a+ \xi. d_j(a),& \, &j=1,2,3 \\
\label{compatible-connection-I2} d_j (\langle\, \xi , \xi'\, \rangle_{\mathcal{A}})&= \langle \nabla_j \, \xi\, , \xi' \rangle_{\mathcal{A}} \, + \, \langle \xi\, , \nabla_j \, \xi' 
\rangle_{\mathcal{A}},&    \forall \, \xi\, , 
\xi' \in \mathcal{E}, &j=1,2,3 \, .
\end{align}
\end{proposition}
\begin{proof} Given a compatible connection $\nabla$ let $\nabla_j=\nabla_{X_j}$ for $j=1,2,3.$. The condition (\ref{compatible-connection-I1},
\ref{compatible-connection-I2}) holds because  $\nabla$ satisfies conditions (\ref{connection-I},\ref{compatibility-I}). Conversely if 
(\ref{compatible-connection-I1}, \ref{compatible-connection-I2}) holds and we define $\nabla$ by specifying its components on the basis \ref{basis} such that $\nabla_{X_j}=\nabla_j$, then clearly the conditions of a compatible connection (\ref{connection-I},\ref{compatibility-I}) are satisfied. 
\end{proof}

For QHM the curvature is given by
\begin{eqnarray*}
 \Theta_\nabla(X_1 \wedge X_3)& = &[\nabla_{X_1}, \nabla_{X_3}],\\
\Theta_\nabla(X_2 \wedge X_3)& = &[\nabla_{X_2}, \nabla_{X_3}],\\
\Theta_\nabla(X_1 \wedge X_2)& = &[\nabla_{X_1}, \nabla_{X_2}]+\frac{1}{\alpha}\nabla_{X_3}.\\
\end{eqnarray*}
Here the third equality uses the relation $[X_1,X_2]=-\frac{1}{\alpha}X_3$ from (\ref{Lie}).
\begin{definition}\rm \label{YM1}
 Let $\mathcal{E}$ be a finitely generated projective $\AAZ$ module with a hermitian structure. Then the Yang-Mills action functional for a compatible connection $\nabla \in C(\mathcal{E})$ is given by
\begin{eqnarray}\label{YM2}
 YM(\nabla)=-\widetilde{\tau}({([\nabla_{X_1},\nabla_{X_3}])}^2 + {([\nabla_{X_2},\nabla_{X_3}])}^2+([\nabla_{X_1},\nabla_{X_2}]+\frac{1}{\alpha}\nabla_{X_3})^2)).
\end{eqnarray}
\end{definition}

\newsection {Yang-Mills for spectral triples}
In (\cite{CON}),  Connes gave a second approach to Yang-Mills for spectral triples.  In this approach one begins with a spectral triple.  Recall that a spectral triple is given by a triple $(\mathcal{A},\mathcal{H},D)$ where 
\begin{enumerate}
\item[(i)] $\mathcal{H}$ is a separable Hilbert space,
\item[(ii)] $\mathcal{A} \subseteq \mathcal{B}(\mathcal{H})$ is a unital involutive sub-algebra closed under holomorphic function calculus,
\item[(iii)] $D$ is a self-adjoint operator with compact resolvent such that $[D,\mathcal{A}] \subseteq \mathcal{B}(\mathcal{H}).$
\end{enumerate}
A spectral triple is $d+$-summable if the Dixmier trace of ${|D|}^{-d}$ is finite. Starting with a $d+$-summable spectral triple Connes defines a complex as follows.
\begin{definition} \rm
Let $(\mathcal{A} ,\mathcal{H} ,D)$ be a $d+$-summable spectral triple. Then the space of universal $k$-forms is given by 
$\Omega^k(\mathcal{A})=\{ \sum_{i=1}^N a_0^{i} \delta a_1^i \ldots \delta a_k^i | n \in \mathbb{N}, a_j^i \in \mathcal{A} \} $. The direct sum of all these spaces $ \Omega^\bullet(\mathcal{A})=\oplus_0^{\infty} \Omega^k(\mathcal{A})$
is the unital graded algebra of universal forms. Here $\delta $ is an abstract linear operator with $\delta^2=0,\delta(ab)=\delta(a)b+a\delta(b)$. $\Omega^\bullet (\mathcal{A})$ becomes a *algebra under the involution ${(\delta a)}^*=-\delta (a^* ) \forall a \in \mathcal{A}$. Let $ \pi: \Omega^\bullet ( \mathcal{A})  \rightarrow \mathcal{B} (\mathcal{H})$ be the $\star$-representation given by $ \pi(a)=a,\pi (\delta a)=[D,a].$
Let $J_k= ker \pi |_{\Omega^k(\mathcal{A})}$ The unital graded differential $\star$-algebra of differential forms
$\Omega^\bullet_D(\mathcal{A})$ is defined by
\[
\Omega_D^\bullet ( \mathcal{A} ) = \oplus_0^\infty \Omega_D^k(\mathcal{A}), \Omega_D^k(\mathcal{A})= \Omega^k(\mathcal{A})/(J_k+\delta J_{k-1}) \cong \pi (\Omega^k(\mathcal{A})/\pi ( \delta J_{k-1}).
\]
The abstract differential $\delta$ induces a differential $\tilde{d}$ on the complex $\Omega^\bullet_D(\mathcal{A})$ so that we get a chain complex $(\Omega^\bullet_D(\mathcal{A}),\tilde{d})$ and a chain map $\pi_D:\Omega^\bullet ( \mathcal{A} ) \rightarrow \Omega^\bullet_D(\mathcal{A})$ such that the following diagram 
\begin{center}
\begin{tikzpicture}[node distance=3cm,auto]
\node (Up)[label=above:$\pi_D$]{};
\node (A)[node distance=1.5cm,left of=Up]{$\Omega^\bullet ( \mathcal{A} )$};
\node (B)[node distance=1.5cm,right of=Up]{$\Omega^\bullet_D(\mathcal{A})$};
\node (Down)[node distance=2cm,below of=Up, label=below:$\pi_D$]{};
\node(C)[node distance=1.5cm,left of=Down]{$\Omega^{\bullet+1} ( \mathcal{A} )$};
\node(D)[node distance=1.5cm,right of=Down]{$\Omega^{\bullet+1}_D(\mathcal{A})$};
\draw[->](A) to (B);
\draw[->](C) to (D);
\draw[->](B)to node{{ $\tilde{d}$}}(D);
\draw[->](A)to node[swap]{{ $\delta$}}(C);
\end{tikzpicture} 
\end{center}
commutes.
\end{definition}
\begin{definition} \rm
Let $\mathcal{E}$ be a  Hermitian, finitely generated projective module over $\mathcal{A}$. A compatible connection on $\,\mathcal{E}\,$ is a linear mapping
$\,\widetilde{\nabla} : \mathcal{E} \longrightarrow \mathcal{E} \, \otimes _\mathcal{A} \Omega _{D}^1\,$ such that,
\begin{enumerate}
\item[(a)] $\widetilde{\nabla}  (\xi a) = (\widetilde{\nabla} \xi)a + \xi \otimes \tilde{d}a, \, \, \, \, \, \forall \xi \in \mathcal{E} , a \in \mathcal{A}$;
\item[(b)] $( \, \xi , \widetilde{\nabla}  \eta \, ) - ( \, \widetilde{\nabla}  \xi , \eta \, ) = \tilde{d}\langle \, \xi , \eta \, \rangle_\mathcal{A}\, \, \, \, \, \, \, \forall \xi , \eta \in \mathcal{E}$.
\end{enumerate}
\end{definition}
The meaning of the last equality in $\Omega_D^1$ is, 
if $ \widetilde{\nabla} (\xi ) = \sum\xi_j\otimes \omega_j $, with $\xi_j \in \mathcal{E}\, , \, \omega_j \in \Omega_D^1(\mathcal{A})$,
then $ (\widetilde{\nabla}  \xi, \eta)=\sum \omega_j^* \langle\xi_j , \eta\rangle_\mathcal{A}$.

Also, any two compatible connections can only differ by an element of 
Hom$_\mathcal{A}(\mathcal{E}\, ,\, \mathcal{E} \otimes_\mathcal{A} \Omega_D^1(\mathcal{A}))$. That is, the space of all
compatible connections on $\,\mathcal{E}$, which we denote by $\widetilde C(\mathcal{E})$, is an affine space with associated vector
space $Hom_\mathcal{A}(\mathcal{E}\, ,\, \mathcal{E} \otimes_\mathcal{A} \Omega_D^1(\mathcal{A}))$.

To define the  curvature $\varTheta$ of a connection $\widetilde{\nabla}$, one first extends $\widetilde{\nabla}$ to a unique linear mapping $\widetilde{\nabla}$ from
$\mathcal{E} \otimes_{\mathcal{A}} \Omega_D^1$ to $\mathcal{E} \otimes_{\mathcal{A}} \Omega_D^2$ such that,
\begin{eqnarray}\label{extended-connection}
\widetilde \nabla (\xi \otimes \omega) = (\widetilde{\nabla} \xi)\omega + \xi \otimes \tilde d\omega, \, \, \, \forall \, \, \xi \in \mathcal{E}, \, \, \omega \in \Omega_D^1.
\end{eqnarray}
It can be easily checked that $\widetilde\nabla$, defined above, satisfies the Leibniz rule.

It follows that $\varTheta = \widetilde\nabla \circ \widetilde{\nabla}$ is an element of $Hom_\mathcal{A}(\mathcal{E},\mathcal{E} \otimes_\mathcal{A} \Omega_D^2)$. Recall that $\Omega_D^2 \cong \pi(\Omega^2)/\pi(dJ_1)$. Let $\mathcal{H}_2$ be the Hilbert space completion
 of $\pi(\Omega^2)$ with respect to the inner-product 
\begin{eqnarray}\label{inner-product}
\langle T_1,T_2\rangle = Tr_\omega(T_1^*T_2|D|^{-d}),\, \forall \, T_1,T_2 \in \pi(\Omega^2). 
\end{eqnarray}
Let $\widetilde{\mathcal{H}}_2$ be the Hilbert space completion of $\pi(dJ_1)$ with the above inner-product. Clearly $\widetilde{\mathcal{H}}_2 \subseteq \mathcal{H}_2$. Let $P$ be the orthogonal projection of $\mathcal{H}_2$ onto the orthogonal complement of  $\widetilde{\mathcal{H}}_2$. Now define $\langle\, [T_1],[T_2]\,\rangle_{\Omega_D^2} = \langle PT_1,PT_2\rangle,\,$ for all $\, [T_i] \in \Omega_D^2$. This gives a well-defined inner-product on $\Omega_D^2$.
Now the inner-product on $Hom_\mathcal{A}(\mathcal{E},\mathcal{E} \otimes_\mathcal{A} \Omega_D^2)$ is described as follows. Suppose $\mathcal{E}=p{\mathcal{A}}^q$, where $p \in M_q(\mathcal{A})$ is a projection.
Then we have the embedding
$$Hom_\mathcal{A}(\mathcal{E},\mathcal{E} \otimes_\mathcal{A} \Omega_D^2) = Hom_\mathcal{A}(pA^q,pA^q \otimes_\mathcal{A} \Omega_D^2) \cong Hom_\mathcal{A}(pA^q,p(\Omega_D^2)^q)\subseteq Hom_\mathcal{A}(\mathcal{A}^q,(\Omega_D^2)^q).$$ The inner product between $\phi,\psi \in Hom_\mathcal{A}(\mathcal{E},\mathcal{E} \otimes_\mathcal{A} \Omega_D^2)$ is given by $$\langle\langle \phi,\psi\rangle\rangle = \displaystyle\sum_{j,k} \langle {(\phi(e_j))}_k,{(\psi(e_j))}_k\rangle_{\Omega_D^2}$$ where $\{e_1,\ldots,e_q\}$ is the standard basis of $\mathcal{A}^q$ and ${(\phi(e_j))}_k$, (respectively ${(\psi(e_j))}_k$ denote the $k$-th component of $\phi(e_j)$ $( \psi(e_j))$. 
\begin{remark} \label{inner-prod}
Let us assume that $\Omega_D^2$ is free of rank $n$ and the inner product described above  between two $n$-tuples $\underline{a},\underline{b} \in \Omega_D^2={\mathcal{A}}^n$ is given by $<\underline{a},\underline{b}>=\sum_j {\rm Tr}_\omega(a_j^* b_j){|D|}^{-d}$. If we use the embedding \[Hom_\mathcal{A}(\mathcal{E},\mathcal{E} \otimes_\mathcal{A} \Omega_D^2)\cong \oplus_{k=1}^n Hom_\mathcal{A}(\mathcal{E},\mathcal{E}) \subseteq \oplus_{k=1}^n M_q(\mathcal{A})\cong \oplus_{k=1}^n \mathcal{A} \otimes M_q(\mathbb{C}). \]
Then $\phi, \psi \in Hom_\mathcal{A}(\mathcal{E},\mathcal{E} \otimes_\mathcal{A} \Omega_D^2)$ can be identified with two $n$-tuples $\phi=(\phi_1,\cdots, \phi_n),$ and $\psi=(\psi_1,\cdots,\psi_n)$, with each $\phi_j,\psi_j \in \mathcal{A} \otimes M_q(\mathbb{C})$. Let $\tau'$ be the trace on $\mathcal{A}$ given by $a \mapsto {\rm Tr}_\omega a {|D|}^{-d}$ then
$$\langle\langle \phi, \psi \rangle \rangle= \sum_{j=1}^n \tau' \otimes {\rm{Trace}} \phi_j^* \psi_j.$$
\end{remark}

\begin{definition} \rm
 The functional on $\widetilde C(\mathcal{E})$ given by $\widetilde{\textit{YM}}\,(\widetilde{\nabla}) = \langle\langle\, \varTheta,\varTheta\, \rangle\rangle$ is called the Yang-Mills functional.
\end{definition}

\newsection{Equivalence of the two approaches}
In this section we will show that for the quantum Heisenberg manifolds there is a correspondence between the set of compatible connections so that the corresponding Yang-Mills functionals agree. To that end one must construct a spectral triple on this algebra. There is a general recipe that begins with $(A,G,\alpha,\tau)$ a $C^*$-dynamical system with an invariant trace. Of course one also requires that the dynamics is governed by a Lie group. Let us assume that the Lie group has dimension $n$. Then by fixing a basis $X_1,\cdots,X_n$ of the Lie-algebra of the Lie group one produces a densely defined operator $D$ on the Hilbert space $\mathcal{H}=L^2(A,\tau) \otimes {\mathbb{C}}^N, N=2^{\lfloor n/2 \rfloor} $. There is a natural representation of the algebra on $\mathcal{H}$ and $D$  produces bounded commutators with the image of $\mathcal{A}$, the smooth algebra of the system. However, in general one does not know whether $D$ admits a self-adjoint extension with compact resolvent. It was shown in (\cite{CS}) that for QHM indeed $D$ admits a self-adjoint extension with compact resolvent provided one chooses the Lie algebra basis considered in (\ref{basis}). For our present purpose it is enough to recall the operators $[D,\phi]$ for $\phi \in S^c$. Note that here the dimension of the associated Lie group is three.
Let $\sigma_1,\sigma_2,\sigma_3$ be $2 \times 2$ self-adjoint trace-less matrices given by
\begin{eqnarray*}
\sigma_1=\left(\begin{matrix} 1& 0 \cr 0 &-1 \cr \end{matrix}\right),
\,  \, \sigma_2=\left(\begin{matrix} 0& -1 \cr -1 &0 \cr \end{matrix}\right),
\, \,  \sigma_3=\left(\begin{matrix} 0& i \cr -i &0 \cr \end{matrix}\right). 
\end{eqnarray*}
Then,
\begin{eqnarray*}
\sigma_1 \sigma_2 =i \sigma_3, \, \, \sigma_2 \sigma_3 = i \sigma_1,\, \, \sigma_3 \sigma_1 = i \sigma_2.
\end{eqnarray*}
Let $\phi \in S^c$, then 
\begin{eqnarray}\label{Dirac-I}
[D,\phi]=\sum \delta_j (\phi) \otimes \sigma_j \mbox{ where }
\delta_j(\phi)= i d_j(\phi)
\end{eqnarray}  and the derivations $d_j$ are given by (\ref{derivation-I},\ref{derivation-II},\ref{derivation-III}).  The $\delta_j$'s satisfy the following commutation relations
\begin{eqnarray}
\label{delta-bracket} [\delta_1,\delta_3]=[\delta_2,\delta_3]=0, [\delta_1,\delta_2]=-\frac{i}{\alpha}\delta_3. 
\end{eqnarray}

These forms were computed in \cite{CS}. In the following proposition we recall the description of the space of forms as $\AAZ-\AAZ$-bimodules.
\begin{proposition} \label{recall-I} \rm \hfill
\begin{enumerate}
\item[(i)]  The space of one forms as an $\AAZ-\AAZ$-bimodule is given by 
\begin{eqnarray*}
  \Omega^1_D ( \AAZ) & =  & \{ \sum a_j \otimes \sigma_j| a_j \in \AAZ , \sigma_j 's  \mbox { as above  } \}  \subseteq \AAZ \otimes M_2(\mathbb{C}) \subseteq \mathcal{B}(\mathcal{H})  \\
    & \cong & \AAZ \oplus \AAZ \oplus \AAZ . 
\end{eqnarray*}
\item[(ii)] $ \pi (\Omega^k (\AAZ))  = \AAZ \otimes M_2(\mathbb{C})= \AAZ \oplus \AAZ \oplus \AAZ \oplus \AAZ.$
\item[(iii)]  $\pi ( \delta  J_1)=\AAZ \otimes I_2 \subseteq \AAZ \otimes M_2(\mathbb{C}) \subseteq \mathcal{B}(\mathcal{H})  $.
\item[(iv)]  The space of two forms as an $\AAZ-\AAZ$-bimodule is given by 
\begin{eqnarray*}
  \Omega^2_D ( \AAZ) & =  & \{ \sum a_j \otimes \sigma_j| a_j \in \AAZ , \sigma_j 's  \mbox { as above  } \}  \subseteq \AAZ \otimes M_2(\mathbb{C}) \subseteq \mathcal{B}(\mathcal{H})  \\
    & \cong & \AAZ \oplus \AAZ \oplus \AAZ . 
\end{eqnarray*}
\item[(v)] The product map from $\Omega^1_D ( \AAZ) \times \Omega^1_D ( \AAZ) $ to $\Omega^2_D ( \AAZ)$ is given by
\[(a \otimes \sigma_j) \cdot (b \otimes \sigma_k)=(1-\delta_{jk})ab \otimes \sigma_j \sigma_k , \forall j,k=1,2,3.\]
Here $\delta_{jk}$ is the Kronecker delta.
\end{enumerate}
\end{proposition}
\begin{proof} Only (v) was not mentioned in (\cite{CS}). This follows because the space of forms $\Omega^1_D ( \AAZ),\Omega^2_D ( \AAZ)$ are identified with subspaces  of $\AAZ \otimes M_2(\mathbb{C})$ and the multiplication is induced from the multiplication on $\AAZ \otimes M_2(\mathbb{C})$.
\end{proof}
We also recall proposition 14 from \cite{CS}.
\begin{proposition}
If $ 1, \hbar \mu, \hbar \nu$ are  independent over $\mathbb Q$ then the positive linear functional on $\AAZ \otimes M_2({\mathbb C})$ given by
$\tau' : a \mapsto tr_\omega a {|D|}^{-3}$ coincides with $\frac{1}{2} (tr_\omega {|D|}^{-3}) \tau \otimes tr$ where $tr_\omega$ is a Dixmier trace. Thus $\tau'=\frac{1}{2} (tr_\omega {|D|}^{-3}) \widetilde{\tau}$, where $\widetilde{\tau}$, is the trace on $End({\mathcal E})$ used in definition (\ref{YM1}).
\end{proposition} 
\begin{proposition} \label{differential} \rm \hfill
\begin{enumerate}
\item[(i)] The differential $\tilde{d}: \AAZ \rightarrow \Omega_D^1(\AAZ)$ satisfies $\tilde{d}(a)=\sum_{j=1}^3 \delta_j(a) \otimes \sigma_j.$
\item[(ii)] The differential $\tilde{d}: \Omega_D^1(\AAZ) \rightarrow \Omega_D^2(\AAZ)$ satisfies
 \begin{eqnarray}
 \tilde{d}(a\otimes \sigma_1)&=&\sum_{j=2,3} \delta_j(a) \otimes \sigma_j \sigma_1,\\
  \tilde{d}(a\otimes \sigma_2)&=&\sum_{j=1,3} \delta_j(a) \otimes \sigma_j \sigma_2,\\
  \label{derivation-I-3}
   \tilde{d}(a\otimes \sigma_3)&=& \delta_1(a) \otimes \sigma_1 \sigma_3 + \delta_2(a) \otimes \sigma_2 \sigma_3 -\frac{i}{\alpha}a \otimes \sigma_1 \sigma_2.
\end{eqnarray}
\end{enumerate}
\end{proposition}
\begin{proof} \rm (i) This follows from $\tilde{d}(a)=[D,a]=\sum_j \delta_j(a) \otimes \sigma_j$.

(ii) The differential $\tilde{d}: \Omega_D^1(\AAZ) \rightarrow \Omega_D^2(\AAZ)$ is defined in such a way that the following diagram commutes.
\begin{center}
\begin{tikzpicture}[node distance=3cm,auto]
\node (Up)[label=above:$\pi_D$]{};
\node (A)[node distance=1.5cm,left of=Up]{$\Omega^1 ( \mathcal{\AAZ} )$};
\node (B)[node distance=1.5cm,right of=Up]{$\Omega^1_D(\mathcal{\AAZ})$};
\node (Down)[node distance=2cm,below of=Up, label=below:$\pi_D$]{};
\node(C)[node distance=1.5cm,left of=Down]{$\Omega^{2} ( \mathcal{\AAZ} )$};
\node(D)[node distance=1.5cm,right of=Down]{$\Omega^{2}_D(\mathcal{\AAZ})$};
\draw[->](A) to (B);
\draw[->](C) to (D);
\draw[->](B)to node{{ $\tilde{d}$}}(D);
\draw[->](A)to node[swap]{{ $\delta$}}(C);
\end{tikzpicture} 
\end{center}
Therefore to  see how it acts on an element of $\Omega^1_D(\mathcal{\AAZ})$ we pick an element and lift that to $\Omega^1(\mathcal{\AAZ})$ and then follow the diagram. Let $\phi_{mn} \in S^c$ be the function $\phi_{m,n}(x,y,p)=e(mx+ny)\delta_{p0}$. These functions are eigenfunctions for $\delta_j$'s and satisfy
\begin{align*}
 \delta_1(\phi_{10})&=2 \pi \phi_{10},& \delta_2 (\phi_{10})&=0,& \delta_3(\phi_{10})&=0,\\
\delta_1(\phi_{01})&=0,& \delta_2 (\phi_{10})&=2 \pi \phi_{01},& \delta_3(\phi_{01})&=0.
\end{align*}
Let $\tilde{a}=\frac{1}{2\pi} a \phi_{10}^* \delta (\phi_{10}) \in \Omega^1$, then,
\begin{align*}
 \pi_D(\tilde{a}) & =  \frac{1}{2\pi} (a \phi_{10}^*\otimes I_2)(\sum_{j=1}^3 \delta_j(\phi_{10}) \otimes \sigma_j))\\
& =  \frac{1}{2\pi} (a \phi_{10}^* 2 \pi \phi_{10}) \otimes \sigma_1\\
& =  a \otimes \sigma_1.\\
\intertext{Therefore,}
 \tilde{d}(a \otimes \sigma_1) & =   \pi_D(\delta(\tilde{a}))\\
& =  \frac{1}{2 \pi}(\sum_{j=1}^3 \delta_j(a \phi_{10}^*) \otimes \sigma_j) ( 2 \pi \phi_{10} \otimes \sigma_1)\\
& =  \sum_{j \ne 1} \delta_j(a) \otimes \sigma_j \sigma_1.\\
\intertext{Similarly}
\tilde{d}(a \otimes \sigma_2) & = \sum_{j \ne 2} \delta_j(a) \otimes \sigma_j \sigma_2.\\
\end{align*}
To see (\ref{derivation-I-3}) observe that,
\begin{eqnarray*}
 \tilde{d}(a \delta_3(b)\otimes \sigma_3) & = & \tilde{d}(a \sum_j \delta_j(b)\otimes \sigma_j)-\tilde{d}(a \delta_1(b)\otimes \sigma_1)-\tilde{d}(a \delta_2(b)\otimes \sigma_2)\\
& = & \tilde{d}(\pi_D(a \delta(b)))-\tilde{d}(a \delta_1(b)\otimes \sigma_1)-\tilde{d}(a \delta_2(b)\otimes \sigma_2)\\
& = & \pi_D(\delta(a)\delta(b))-\tilde{d}(a \delta_1(b)\otimes \sigma_1)-\tilde{d}(a \delta_2(b)\otimes \sigma_2)\\
& = & \sum (\delta_j(a) \otimes \sigma_j ) (\delta_k(a) \otimes \sigma_k ) -\tilde{d}(a \delta_1(b)\otimes \sigma_1)-\tilde{d}(a \delta_2(b)\otimes \sigma_2) \mbox { mod } \pi(\delta J_1)\\
& = & \delta_1(a) \delta_3 (b) \otimes \sigma_1 \sigma_3 +\delta_2(a) \delta_3 (b) \otimes \sigma_2 \sigma_3 +a[\delta_2,\delta_1](b) \otimes \sigma_1 \sigma_2 \\
&& -a \delta_3 (\delta_1(b)) \otimes \sigma_3 \sigma_1 -a \delta_3 (\delta_2(b)) \otimes \sigma_3 \sigma_2\\
& = & \delta_1(a \delta_3(b)) \otimes \sigma_1 \sigma_3 +
\delta_2(a \delta_3(b)) \otimes \sigma_2 \sigma_3 +a[\delta_2,\delta_1](b) \otimes \sigma_1 \sigma_2 \\
&=&  \delta_1(a \delta_3(b)) \otimes \sigma_1 \sigma_3 +
\delta_2(a \delta_3(b)) \otimes \sigma_2 \sigma_3
+\frac{i}{\alpha} a \delta_3(b) \otimes \sigma_1 \sigma_2.
\end{eqnarray*}
The last equality uses $[\delta_2,\delta_1]=\frac{i}{\alpha}  \delta_3.$
Since span of elements of the form $a\delta_3(b)$ forms an ideal in $\AAZ$ and $\AAZ$ is simple, (\ref{derivation-I-3}) follows.
\end{proof}
\begin{corollary}
$\tilde{d}(1 \otimes \sigma_j)= \begin{cases} 0 & \mbox { if } j = 1,2;\\
\frac{
-1}{\alpha}  \otimes \sigma_3 & \mbox{ if } j=3. \end{cases} $
\end{corollary}
Now we have all the ingredients to describe the space $\widetilde{C}({\mathcal E})$ of compatible connections on a finitely 
generated projective $\AAZ$-module ${\mathcal E}$ with a Hermitian structure.
\begin{proposition}\label{compatible-connection-II} Let $\mathcal E$ be a finitely generated projective $\AAZ$ module. Then  the space
 $\widetilde{C}(\mathcal{E})$ 
of compatible  connections is given by triples of linear maps $\widetilde{\nabla}_j : \mathcal{E} \rightarrow \mathcal {E}, j=1,2,3 $ such that
\begin{align} \label{compatible-connection-II1}
\widetilde{\nabla}_j(\xi.a) &= \nabla_j(\xi).a+ \xi. \delta_j(a),& \, &j=1,2,3 \\
\label{compatible-connection-II2} \delta_j (\langle\, \xi , \xi'\, \rangle_{\mathcal{A}})&=  \langle \xi\, , \widetilde{\nabla}_j \, 
\xi' \rangle-\langle \widetilde{\nabla}_j \, \xi\, , \xi' \rangle \,  \,
,&    \forall \, \xi\, , 
\xi' \in \mathcal{E}, &j=1,2,3 \, .
\end{align}
\end{proposition}
\begin{proof} 
By proposition (\ref{recall-I}) we can identify ${\mathcal E}\otimes_{\AAZ} \Omega_D^1(\AAZ)$ with the 
subspace $\sum_j {\mathcal E} \otimes \sigma_j \subseteq {\mathcal E} \otimes M_2({\mathbb C})$. Thus any compatible connection 
$\widetilde{\nabla}$ is prescribed by 
three maps $\widetilde{\nabla}_j : {\mathcal E} \rightarrow {\mathcal E}$ such that 
\[\widetilde{\nabla}(\xi)=\sum_{j=1}^3 \widetilde{\nabla}_j(\xi) \otimes \sigma_j.\]
Then
\begin{eqnarray*}
\widetilde{\nabla}(\xi.a) &=&  \sum_{j=1}^3  \widetilde{\nabla}_j(\xi.a) \otimes \sigma_j \\
&=&  \widetilde{\nabla}(\xi).a + \xi \otimes \widetilde{d}(a) \\
&=&  \sum_{j=1}^3  \widetilde{\nabla}_j(\xi).a \otimes \sigma_j + \sum_{j=1}^3 \xi.\delta_j(a) \otimes \sigma_j .\\
\end{eqnarray*}
Thus comparing coefficients of $\sigma_j$ we get,
\[\widetilde{\nabla}_j(\xi.a) = \nabla_j(\xi).a+ \xi. \delta_j(a), \, j=1,2,3.\]
For (\ref{compatible-connection-II2}) note that,
\begin{eqnarray*}
\sum_{j=1}^3 \delta_j(\langle \xi, \xi' \rangle ) \otimes \sigma_j &=& \widetilde{d}\langle \xi,\xi' \rangle) \\
&=& (\xi, \widetilde{\nabla}\xi')-(\widetilde{\nabla} \xi, \xi') \\
 &=&\sum_{j=1}^3 (\langle \xi , \widetilde{\nabla}_j \xi' \rangle-\langle \widetilde{\nabla}_j \, \xi\, , \xi' \rangle) \otimes \sigma_j . \\
\end{eqnarray*}
\end{proof}
\begin{theorem} \rm Let $\mathcal{E}$ be a finitely generated projective $\AAZ$ module with a Hermitian structure. Then 
$\Phi:C(\mathcal{E}) \rightarrow \widetilde{C}({\mathcal{E}})$ given by $\Phi(\nabla)=\widetilde{\nabla}$, where 
$\widetilde{\nabla}(\xi)=i\nabla(\xi)$ is well defined and $$\frac{1}{2} (tr_\omega {|D|}^{-3})\rm{YM}(\nabla)=\widetilde{\rm{YM}}(\Phi(\nabla)).$$
\end{theorem}
\begin{proof} \rm Let $\nabla \in C(\mathcal{E})$ be a compatible connection and  $\nabla_j,j=1,2,3$ be its components  as defined in the proof of proposition (\ref{compatible-connection-I}). If we define ${\widetilde{\nabla}}_j=i \nabla_j, j=1,2,3$, then ${\widetilde{\nabla}}_j$'s satisfy (\ref{compatible-connection-II1}) because $\delta_j=i.d_j,j=12,3$ and (\ref{compatible-connection-I1}) holds. Similarly (\ref{compatible-connection-II2}) follows from (\ref{compatible-connection-I2}). Thus the triple ${\widetilde{\nabla}}_j,j=1,2,3$ defines a compatible connection $\widetilde{\nabla} 
\in \widetilde{C}({\mathcal{E}})$. This proves the map $\Phi$ is well defined with the given domain and range. In fact it is an isomorphism.
Let $\widetilde{\nabla}$ denote the extended connection as defined on (\ref{extended-connection}) then using propositions, (\ref{recall-I},\ref{differential}), we get,
\begin{eqnarray*}
\widetilde{\nabla}(\xi \otimes \sigma_1) & = & \sum_{j\ne 1} {\widetilde{\nabla}}_j(\xi) \otimes \sigma_j \sigma_1 \\
\widetilde{\nabla}(\xi \otimes \sigma_2) & = & \sum_{j\ne 2} {\widetilde{\nabla}}_j(\xi) \otimes \sigma_j \sigma_2 \\
\widetilde{\nabla}(\xi \otimes \sigma_3) & = & \sum_{j\ne 3} {\widetilde{\nabla}}_j(\xi) \otimes \sigma_j \sigma_3 -\frac{1}{\alpha} \xi \otimes \sigma_3. \\
\end{eqnarray*} 
The curvature $\varTheta$ of the connection $\widetilde{\nabla}$ is given by $\varTheta=\widetilde{\nabla} \circ \widetilde{\nabla}$. Which turns out to be,
\[\varTheta(\xi)=i[{\widetilde{\nabla}}_2,{\widetilde{\nabla}}_3](\xi)\otimes \sigma_1 + i [{\widetilde{\nabla}}_3,{\widetilde{\nabla}}_1](\xi) \otimes \sigma_2 +(i[{\widetilde{\nabla}}_1,{\widetilde{\nabla}}_2]-\frac{1}{\alpha}{\widetilde{\nabla}}_3)(\xi) \otimes \sigma_3.\]
Repeated application of \ref{compatible-connection-II2} gives,
\begin{eqnarray}
\label{id1} \delta_k(\langle \xi, {\widetilde{\nabla}}_j (\eta) \rangle &=& \langle \xi, {\widetilde{\nabla}}_k ({\widetilde{\nabla}}_j (\eta)) \rangle -\langle {\widetilde{\nabla}}_k (\xi), {\widetilde{\nabla}}_j (\eta) \rangle,\\
\label{id2} \delta_j(\langle \xi, {\widetilde{\nabla}}_k (\eta) \rangle &=& \langle \xi, {\widetilde{\nabla}}_j ({\widetilde{\nabla}}_k (\eta)) \rangle -\langle {\widetilde{\nabla}}_j (\xi), {\widetilde{\nabla}}_k (\eta) \rangle, \\
\label{id3} \delta_j(\langle {\widetilde{\nabla}}_k (\xi), \eta \rangle &=& \langle {\widetilde{\nabla}}_k (\xi), {\widetilde{\nabla}}_j (\eta) \rangle -\langle {\widetilde{\nabla}}_j ({\widetilde{\nabla}}_k (\xi )), \eta \rangle ,\\
\label{id4} \delta_k(\langle {\widetilde{\nabla}}_j (\xi), \eta \rangle &=& \langle {\widetilde{\nabla}}_j (\xi), {\widetilde{\nabla}}_k (\eta) \rangle
-\langle {\widetilde{\nabla}}_k ({\widetilde{\nabla}}_j (\xi )), \eta \rangle . 
\end{eqnarray}
Now, (\ref{id1})-(\ref{id2})+(\ref{id3})-(\ref{id4}) gives,
\begin{eqnarray}
\label{nabla*} \langle \xi, [{\widetilde{\nabla}}_k,{\widetilde{\nabla}}_j](\eta) \rangle - \langle [{\widetilde{\nabla}}_j,{\widetilde{\nabla}}_k](\xi), \eta \rangle &=& [\delta_k,\delta_j]\langle \xi, \eta \rangle).\end{eqnarray}
Combining (\ref{delta-bracket}) and (\ref{nabla*}) we get,
\begin{eqnarray*}
\langle \xi, [{\widetilde{\nabla}}_1,{\widetilde{\nabla}}_3](\eta) \rangle &=& \langle [{\widetilde{\nabla}}_3,{\widetilde{\nabla}}_1](\xi), \eta \rangle \\
\langle \xi, [{\widetilde{\nabla}}_2,{\widetilde{\nabla}}_3](\eta) \rangle &=& \langle [{\widetilde{\nabla}}_3,{\widetilde{\nabla}}_2](\xi), \eta \rangle \\
\langle \xi, (i[{\widetilde{\nabla}}_1,{\widetilde{\nabla}}_2]-\frac{1}{\alpha}{\widetilde{\nabla}}_3)(\eta) \rangle &=& \langle (i[{\widetilde{\nabla}}_1,{\widetilde{\nabla}}_2]-\frac{1}{\alpha}{\widetilde{\nabla}}_3)(\xi), \eta \rangle 
\end{eqnarray*}
These relations give,
\begin{eqnarray*} \widetilde{\textit{YM}}\,(\widetilde{\nabla}) &=& \langle\langle\, \varTheta,\varTheta\, \rangle\rangle  \\
& = &  \frac{1}{2} (tr_\omega {|D|}^{-3}) \widetilde{\tau}( -{([{\widetilde{\nabla}}_1,{\widetilde{\nabla}}_3])}^2-{([{\widetilde{\nabla}}_2,{\widetilde{\nabla}}_3])}^2+{(i[{\widetilde{\nabla}}_1,{\widetilde{\nabla}}_2]-\frac{1}{\alpha}{\widetilde{\nabla}}_3)}^2 )\\
& = & -\frac{1}{2} (tr_\omega {|D|}^{-3}) \widetilde{\tau}({([\nabla_{1},\nabla_{3}])}^2 + {([\nabla_{2},\nabla_{3}])}^2+([\nabla_{1},\nabla_{2}]+\frac{1}{\alpha}\nabla_{3})^2))\\
&=& \frac{1}{2} (tr_\omega {|D|}^{-3}) \textit{YM}(\nabla).\\
\end{eqnarray*}
\end{proof}
 

\begin{thebibliography}{99}
 \rm
\bibitem{CS} Chakraborty, Partha Sarathi; Sinha, Kalyan B. : Geometry on the quantum Heisenberg manifold. {\it  J. Funct. Anal.}, {\bf   203},  (2003),  no. 2, 425-452.
\bibitem{Co1} A. Connes, {\it $C^*$-alg\`{e}bres et g\'{e}om\'{e}trie differentielle}, C.R. Acad. Sc. Paris Ser. \textbf{A}-\textbf{B} 290 (1980), no. 13, A599$-$A604.
\bibitem{CR} A. Connes, M.A. Rieffel, {\it Yang-Mills for non-commutative two-tori}, Contemp. Math. 62 (1987) 237-266.
\bibitem{CON} Connes, A.  : {\it Noncommutative Geometry}, Academic
  Press (1994).
\bibitem{KANG} Kang, Sooran : The Yang-Mills functional and Laplace's equation on quantum Heisenberg manifolds. {\it  J. Funct. Anal.},  {\bf 258},  (2010),  no. 1, 307-327.
\bibitem{RI1} Rieffel, M. Deformation quantization of Heisenberg manifolds,{\it Communications in Math. Phys.},{\bf 122}, 531-562,(1989)
\end{thebibliography}
\end{document}